\documentclass[12pt, reqno]{amsart}
\usepackage{amsmath}

\usepackage{amsfonts}
\usepackage{amssymb}
\usepackage{comment}
\usepackage{graphicx}
\usepackage{xcolor}
\usepackage{enumitem}
\usepackage[margin=1in,footskip=0.25in]{geometry}
\providecommand{\U}[1]{\protect\rule{.1in}{.1in}}
\newtheorem{theorem}{Theorem}

\newtheorem{corollary}[theorem]{Corollary}

\newtheorem{definition}[theorem]{Definition}

\newtheorem{lemma}[theorem]{Lemma}

\newtheorem{proposition}[theorem]{Proposition}
\newtheorem{remark}[theorem]{Remark}

\begin{document}

\title[Heisenberg uncertainty for lca groups]{Heisenberg uncertainty inequalities for locally compact abelian groups}


\author{Hartmut F\"uhr}

\address{Lehrstuhl f\"ur Geometrie und Analysis, RWTH Aachen University, D-52056 Aachen,
Germany}

\email{fuehr@mathga.rwth-aachen.de}

%
%
%
%




\begin{abstract}
We prove a version of Heisenberg's uncertainty principle for a rather general class of locally compact abelian groups. We compare the lower bound provided by our approach with the optimal lower bound in the Euclidean case, and formulate the Heisenberg uncertainty principle for local fields. 
\end{abstract}




\maketitle

\section{Introduction}\label{sec1}

Heisenberg's uncertainty principle for the real line states that, for all functions $f \in L^2(\mathbb{R})$
\begin{equation} \label{eqn:HI}
\inf_{a,b \in \mathbb{R}}
\int_{\mathbb{R}} |x-a|^2 |f(x)|^2 dx \int_{\mathbb{R}}|\xi-b|^2 |\widehat{f}(\xi)|^2 d\xi \ge 
\frac{\| f \|_2^4}{16 \pi^2}~.
\end{equation}
Here we used the notation 
\[
\widehat{f}(\xi) = \int_{\mathbb{R}} f(x) e^{-2\pi i \xi x} dx
\] for the Fourier transform $\widehat{f}$ of $f \in L^1(\mathbb{R})$, suitably extended to the Plancherel transform on $L^2(\mathbb{R})$.

To cite the informal interpretation from \cite{MR1448337}: 
\textit{A function and its Fourier transform cannot both be sharply localized.}
Subsequent to Heisenberg's inequality, various versions of the uncertainty principle were conceived, using different notions of concentration. In order to keep this introduction on a scale that is compatible with the overall length of the paper, we will only name (but not formulate) Hardy's uncertainty principle \cite{MR1574130}, the qualitative uncertainty principle \cite{MR780328} and the Donoho-Stark uncertainty principle \cite{MR997928} as representatives of a fairly long list of possible alternative meanings, each of which resulted in uncertainty principles of varying scope. Overviews (by no means complete), which cover these and other results, can be found in \cite{MR1448337,MR3265736,MR4229152}. To this day uncertainty principles and related ideas of measures of concentration permeate various areas of mathematics, such as mathematical signal processing, approximation theory, or the theory of compressed sensing. 

Each of the uncertainty principles explicitly mentioned above has been generalized to more general locally compact groups, even beyond the abelian setting; see e.g. \cite{MR2587584,MR3336090,MR4744878,MR1353372,MR1181081,MR2122203} for a small, subjective sample of the relevant literature. 
Group-theoretic generalizations of Heisenberg's uncertainty principle mostly stayed within the realm of Lie groups; see e.g. \cite{MR2355602,MR3336090}. To our knowledge, generalizations to more general lca groups have been largely unexplored, with the notable exception \cite{MR2126447}.

This paper shows that many locally compact abelian groups (from now on abbreviated to lca groups)  allow a version of Heisenberg's uncertainty principle, where time-frequency concentration is quantified by integrating squared distances against the modulus squared of the function and its Fourier transform, where the distance are measured by suitable choices of invariant metrics on the group and its dual. 

\section{Notations and statement of the main theorem}

We let $G$ denote a locally compact abelian group and $\widehat{G}$ its character group. Both $G$ and $\widehat{G}$ will be written multiplicatively; the neutral elements of both groups are denoted as $e$ (by slight abuse of notation). Haar integration of a function on $G$ is written as 
\[
\int_G f(x) dx~,
\] and analogously for functions on $\widehat{G}$. Given a Borel subset $A \subset G$, $|A|$ denotes its Haar measure, i.e. $|A| = \int_A 1 dx$. The same notation will be employed for subsets $ B \subset \widehat{G}$. We let $L^2(G)$ denote the $L^2$-space of $G$ with respect to Haar measure. Given $f \in L^2(G)$, we let $\widehat{f}$ denote its Plancherel transform, and we assume that the Haar measures of $G$ and $\widehat{G}$ are normalized to guarantee the equality $\| f \|_2^2 = \| \widehat{f} \|_2^2$ for all $f \in L^2(G)$. 

Let $d$ and $\widehat{d}$ denote suitable metrics on $G$ and $\widehat{G}$, respectively. We will denote open balls in $G$ by 
\[
B_r^d(x) = \{ y \in G : d(x,y) < r \}~;
\] We will omit the superscript $d$ whenever it is clear from the context. The analogously defined open balls in $\widehat{G}$ are denoted by $\widehat{B}_r(\xi)$.
The following definition formulates a list of basic conditions on the metrics. While our main result will concern metrics, it will be useful for the following to discuss the larger classes of \textit{pseudometrics}, which allows $d(x,y) = 0$ to hold for distinct $x,y$

\begin{definition} \label{defn:main}
Given a locally compact abelian group $G$, and a pseudometric $d: G \times G \to \mathbb{R}$.
\begin{enumerate}
    \item[(i)] $d$ is called \textbf{invariant} if $\forall x,y,z \in G$ one has $d(xy,xz) = d(y,z)$.
    \item[(ii)] $d$ is called \textbf{proper} if for all $r>0$ the open balls $B_r(e) = \{ x \in G : d(x,e) < r \}$ are relatively compact. 
    \item[(iii)] $d$  \textbf{has polynomial growth} if there exist $a,A,t>0$ such that
 \begin{equation}  \label{eqn:pg} \forall r>0\,\,\,:\,\,\|B_r(0)| \le a + A r^t \end{equation}
   $d$ has \textbf{homogeneous polynomial growth} if (\ref{eqn:pg}) holds with $a=0$.
\end{enumerate}
\end{definition}

\begin{remark}
\begin{enumerate}
\item[(a)] Since $|B_r(0)| \ge |\{ e \}|$ for all $r>0$, discrete groups cannot have homogeneous polynomial growth.
\item[(b)] A finitely generated group is called \textbf{of polynomial growth} if the word metric with respect to any finite system of generators has polynomial growth. 

It is well-known that all finitely generated abelian groups have polynomial growth. However, there also exist discrete abelian groups allowing invariant metrics of polynomial growth that are not finitely generated; see Remark \ref{ex:non_fg}.
\end{enumerate}
\end{remark}


With these definitions and notations in place, we can now formulate the main theorem: 
\begin{theorem} \label{thm:main}
Let $G$ be a locally compact abelian group, with character group $\widehat{G}$. Let $d,\widehat{d}$ denote continuous invariant proper metrics with homogeneous polynomial growth of order $t$, i.e. fulfilling
\[
\forall r > 0~:~ |B_r(e)| \le A r^t ~,~ |\widehat{B}_r(e)| \le \widehat{A} r^t~,
\] with suitable constants $t,A, \widehat{A}>0$. Then all $f \in L^2(G)$ fulfill the inequality \begin{equation} \label{eqn:Heisenberg}
\inf_{a \in G, b \in \widehat{G}} \left( \int_G d(x,a)^2 |f(x)|^2 dx \int_{\widehat{G}} \widehat{d}(\xi,b)^2 |\widehat{f}(\xi)|^2 d\xi \right) \ge C_G \|f\|_2^4~.
\end{equation}
Here $C_G$ is given by 
\begin{equation} \label{eqn:def_C_G}
C_G = \frac{\alpha_t}{A^{2/t} \widehat{A}^{2/t}}~,~ \alpha_t = \max_{\delta \in (0,1/2)} (1-2 \delta)^{4/t} (2 \delta - \delta^2)^4~.
\end{equation}
\end{theorem}

Note that the assumptions in Theorem \ref{thm:main} on $G$ and $\widehat{G}$, as well as its conclusion, are perfectly symmetric with respect to $G$ and $\widehat{G}$.

The following propositions clarify the availability of metrics needed to formulate the uncertainty principle. A useful tool for this purpose will be a structure theorem \cite[Corollary 3]{MR564688}: Every lca group $G$ is topologically isomorphic to a direct product, $G \cong \mathbb{R}^m \times H$, where $H$ is an lca group possessing an open compact subgroup $K$.

The first result looks at $G$ in isolation. 
\begin{proposition} \label{prop:ex_metric}
Let $G$ denote a second countable lca group, $G \cong \mathbb{R}^n \times H$, and $H$ has a compact open subgroup $K$. Then $G$ admits an invariant metric with homogeneous polynomial growth if $H$ is nondiscrete and $H/K$ admits an invariant metric of polynomial growth.
\end{proposition}

\begin{proof}
Assume that $H/K$ admits an invariant metric with polynomial growth of order $t$. Pulling this map back via the quotient map $H \to H/K$, we obtain a pseudometric $d_1$ on $A$ of polynomial growth. Since $H/K$ is discrete, all metric balls in $H/K$ are finite, hence there exists a uniform lower bound $c < d_1(x,y)$ for all $x,y \in H, x^{-1}y  \not\in K$. After rescaling $d_1$, if necessary, we may assume $c=1$. 

Next assume that $A$ is nondiscrete, or equivalently, that $K$ is compact and nondiscrete. We assume that its Haar measure is normalized by $|K| = 1$. There exists a sequence $(U_n)_{n \in \mathbb{N}_0}$ of open, relatively compact neighborhoods $U_n \subset K$ of the identity satisfying the following conditions: $U_{0} = K$, 
\[ \forall n \in \mathbb{N}_0 \,\,:\,\, U_n = U_n^{-1}\,,\, U_{n+1}^2 \subset U_n\,, \]
and in addition $\bigcap_{n \in \mathbb{N}} U_n = \{ e \}$ and 
\[
\forall n \in \mathbb{N}_0: |U_n| \le 2^{-nt}~.
\]
By a well-known metrization result (e.g. \cite[Theorem 8.2]{MR551496}, there exists a continuous, invariant metric $d_K$ on $K$ satisfying
\[
B^{d_K}_{2^{-n}}(e) \subset U_{n}\,\,,
\] in particular $d_K \le 1$. 
But this entails 
\[
|B^{d_K}_r(e)| \le|U_| \le r^t \,\,  
\] for all $r = 2^{-n}$, and then
\[ |B^{d_K}_r(e)| \le 2r^t \] for all $0 < r < \infty$.

Hence $d_K$ has homogeneous polynomial growth of order $t$. Since $d_K$ is bounded by one, we can define 
\[
d_2(x,y) = \left\{ \begin{array}{cc} d_K(x^{-1}y,e) & \mbox{ if } x^{-1}y \in K \\ 1 & \mbox{otherwise} \end{array} \right.
\]
and obtain an invariant metric on $d_2$ on $H$ that extends $d_K$. One then readily verifies that 
$d_H = \max(d_1,d_2)$ defines a metric on $H$ with homogeneous polynomial growth of order $t$: For $0<r < 1$, a homogeneous estimate is provided by the observation
$B_r^{d_H}(e) = B_r^{d_K}(e)$, whereas for $r\ge1$, 
\[ B_r^{d_A}(e) = B_r^{d_1}(e) \,\,\] and the inhomogenous estimate available for $|B_r^{d_1}(e)|$ implies a homogeneous one of the same growth order holding uniformly in $[1,\infty)$. 

We define the invariant metric $d_G$ on $G = \mathbb{R}^n \times H$ by
\begin{equation} \label{eqn:prod_met} d_G ((u,x),(v,y)) = \max (|u-v|, d_H(x,y) ) \,\,,\end{equation}
leading to $B_{r}^{d_G}(e) = B_r(0) \times B_r^{d_H}(e)$, and thus $d_G$ is of homogeneous growth of order $t+m$.
\end{proof}
The following result formulates conditions on $G$ that guarantee the existence of a pair of metrics giving rise to an uncertainty inequality. 
\begin{corollary}
Let $G$ be an lca group with dual group $\widehat{G}$. Let $G \cong \mathbb{R}^d \times A$, where $A$ is an lca group possessing a compact open subgroup $K$. Assume that $A$ is nondiscrete, and that both discrete groups $A/K$ and $\widehat{K}$ admit invariant metrics with polynomial growth. Then there exists a pair $d,\widehat{d}$ of metrics on $G$ and $\widehat{G}$ respectively, that fulfill the requirements of Theorem \ref{thm:main}.
\end{corollary}
\begin{proof}
    $G \cong \mathbb{R}^n \times A$ entails $\widehat{G} \cong \mathbb{R}^n \times \widehat{A}$. If $K < A$ is a compact open subgroup, let $K^\bot \subset \widehat{A}$ denote the \textit{annihilator} of $K \i \widehat{G}$, i.e. the subgroup of characters of $A$ that are duality trivial on $K$. Pontryagin-van Kampen duality yields $\widehat{A/K} \cong K^\bot$. Furthermore,  the exact sequence
    \[
    1 \to K \to A \to A/K \to 1 
    \]
    gives rise to the exact sequence
    \[
    1 \to K^\bot \to \widehat{A} \to \widehat{K} \to 1\,\,.
    \] Here $\widehat{A/K}$ is compact because $A/K$ is discrete; and $\widehat{K} \cong \widehat{A}/\widehat{A/K}$ is discrete, because $K$ is compact. Hence $K^\bot < \widehat{G}$ is an open compact subgroup with $\widehat{A}/K^\bot \cong \widehat{K}$. 

    By assumption there exist invariant metrics with associated polynomial growth estimates on the discrete groups $A/K$ and $\widehat{K}$. Since the growth estimates on discrete groups are necessarily inhomogeneous, there exist  growth estimates for $A/K, \widehat{K}$ with a \textit{common} growth rate $t$ (e.g. by taking the maximum of both rates). 
    Hence the proof of Proposition \ref{prop:ex_metric} provides metrics $d, \widehat{d}$ that both have homogeneous polynomial growth with rate $t+m$.
\end{proof}

\begin{remark} \label{rem:minimizer}
    For every proper continuous metric $d$, and every $f \in L^2(G)$ such that
    \[
    \int_G |x| |f(x)|^2 dx < \infty~,
    \] the map 
    \[
    V_f : G \ni a \mapsto \int_G |x-a|^2 |f(x)|^2 dx 
    \] is well-defined and continuous with $\lim_{a \to \infty} V_f(a)= \infty$. Hence a simple compactness argument yields that $V_f$ has a minimizer $a_0$, i.e. 
    \[
    \inf_{a \in G}  \int_G |x-a|^2 |f(x)|^2 dx =  \int_G |x-a_0|^2 |f(x)|^2 dx~.
    \] In fact, in the Euclidean case it is easy to show that 
    \[
    a_0 = \int_{\mathbb{R}^n} x |f(x)|^2 dx/ \| f \|_2^2
    \] is the unique minimizer. Note that in this case $a_0$ is nothing but the mean of the random variable with probability density given by $|f(x)|^2/\| f \|_2^2$.

    This process of determining minimizers is not readily available in more general locally compact groups, nor is the minimizer $a_0$ necessarily unique (see Example \ref{ex:counter_lf} below).  
\end{remark}

\begin{remark}
    Note that the homogeneous polynomial growth condition on $d$ and $\widehat{d}$ in Theorem \ref{thm:main} excludes discrete groups $G$, and compact groups by duality. However, the direct product $G = A \times K$ of an infinite discrete group $A$ with polynomial growth with a nondiscrete compact group $K$ such that $\widehat{K}$ falls in the realm of Theorem \ref{thm:main}.

    In particular, while our results do not allow to establish an uncertainty inequality for either $\mathbb{Z}$ or $\mathbb{T}$, they apply to $\mathbb{Z}^m \times \mathbb{T}^n$ for all $m,n>0$.
\end{remark}

\section{Proof of Theorem \ref{thm:main}}

Throughout this section $G$ is assumed to be a second-countable lca group; then $\widehat{G}$ shares these properties \cite{MR1038803}.
Furthermore,  $d,\widehat{d}$ are assumed to be continuous, invariant proper metrics of homogeneous polynomial growth on $G$ and $\widehat{G}$, respectively. 

Given $r>0$, let $P_r : L^2(G) \to L^2(G)$ denote the orthogonal projection map $f \mapsto f \cdot \mathbf{1}_{B_r(e)}$, where $\mathbf{1}_A$ denotes the indicator function associated to the measurable set $A$. The operator $\widehat{P}_r: L^2(\widehat{G}) \to L^2(\widehat{G})$ is analogously defined (with $\widehat{d}$ replacing $d$), and we define 
\[
Q_r = \mathcal{F}^{-1} \circ \widehat{P}_r \circ \mathcal{F}~,
\]
where $\mathcal{F} : L^2(G) \to L^2(\widehat{G})$ denotes the Plancherel transform. 
Furthermore, let 
\[
T: L^2(G) \to L^2(G), (Tf)(x) = \delta(x,e) f(x) ~. 
\] This operator is usually unbounded (unless $G$ has finite diameter), but it is  selfadjoint if taken with its natural domain $D(T) = \{ f : d(\cdot,e) f \in L^2(G) \}$. Note however that these aspects will not be of importance for the following. 

We let $\widehat{T}(f)$ denote the multiplication operator on the Fourier transform side, 
\[
\widehat{T} (f) = \mathcal{F}^{-1}(\widehat{d}(\cdot,e) \mathcal{F}(f))~.
\]

Observe that the operators $T$ and $\widehat{T}$ allow to rewrite the left-hand side of the uncertainty principle as 
\[
\int_G d(x,e)^2 |f(x)|^2 dx \int_{\widehat{G}} \widehat{d}(\xi,e)^2 |\widehat{f}(\xi)|^2 d\xi = \| T f\|_2^2 \|\widehat{f} \|_2^2~. 
\]
The overall proof strategy to obtain a lower bound is vaguely similar to the one employed in \cite{MR3336090}, relating norm estimates on $\| Tf \|_2 $ and $\| \widehat{T} f\|_2$ to concentration estimates for $f, \widehat{f}$ respectively, and then employing a well-known bound for the product of the latter. 

We first need a result concerning measure disintegration.
\begin{lemma}
    Let $d$ denote a continuous metric on $G$. Then Haar measure disintegrates over the spheres, i.e. there exists a family of Borel measures $(\mu_r)_{ r \ge 0}$ with each $\mu_r$ supported in $S_r(e)$, as well as a Borel measure $\nu$ on $\mathbb{R}_0^+$, such that for each nonnegative Borel function $f: G \to \mathbb{R}_0^+$ one has the equality 
    \[
    \int_G f(x) dx = \int_{\mathbb{R}_0^+} \int_{S_r(e)} f(x) d\mu_r(x) d\nu(r)  
    \]
\end{lemma}
\begin{proof}
   As $G$ is second-countable, it is completely metrizable. Furthermore, there exists a function $f: G \to \mathbb{R}^+$ with $\int_G f(x) dx = 1$. Hence $\mu = f dx$ is a probability measure on $G$, and $p: G \ni x \mapsto d(x,e) \in \mathbb{R}$ is a continuous map  between two completely metrizable spaces. 
   Now the measure disintegration theorem \cite[Theorem 5.3.1]{MR2401600} shows the existence of a measure disintegration $(\tilde{\mu}_r)_{ r \ge 0}$, $\nu$ of $\mu$, with $\nu$ being the image measure of $\mu$ under $p$, and $\tilde{\mu}_r$ supported in $S_r(e)$. But then $d\mu_r(x) = \frac{1}{f(x)} d\tilde{\mu}_r(x)$ together with $\nu$ provides the desired measure disintegration of Haar measure. 
\end{proof}


The following proposition formulates the concentration estimate, stating that any normalized function $f$ with $\| T f \|_2< \infty$ is essentially concentrated in a ball of radius comparable to $\| T f \|_2$.
\begin{proposition} \label{prop:main}
Let $\delta \in (0,1/2)$ be given, as well as $f \in L^2(G)$ with $\| f \|_2 = 1$ and $\| Tf \|_2 < \infty$. Then there exists $r_0 >0$ satisfying 
\[
\forall r > r_0: \| P_{r} f \|_2 \ge 1 - \delta~, \mbox{ and } r_0 \le \frac{\| T f \|_2}{\left(2 \delta - \delta^2 \right)^{1/2}}~.
\]
\end{proposition}
\begin{proof} The core idea of the proof is to derive a concentration estimate by applying Markov's inequality to a suitably defined random variable. The measure decomposition over the $d$-spheres allows to write
\begin{eqnarray*}
 1 & = & \int_G |f(x)|^2 dx \\
  & = & \int_{\mathbb{R}_0^+} \int_{S_r(e)} |f(x)|^2 d\mu_r(x) d\nu(r)~,
\end{eqnarray*}
hence letting  
\[
d\nu_f(r) = \int_{S_r(e)} |f(x)|^2 d\mu_r(x) d\nu(r)
\] defines a probability measure $\nu_f(r)$ on $\mathbb{R}_0^+$. From now on let $R$ denote a random variable with probability given by $\nu_f$. We then obtain
\begin{eqnarray}
 \nonumber    \mathbb{E}(R^2) & = & \int_{\mathbb{R}_0^+} r^2 \int_{S_r(e)} |f(x)|^2 d\mu_r(x) d\nu(r) \\
 \nonumber   & = & \int_{\mathbb{R}_0^+} \int_{S_r(e)} d(x,e)^2 |f(x)|^2 d\mu_r(x) d\nu(r) \\
 \label{eqn:expectation}   & = & \int_G d(x,e)^2 |f(x)|^2 dx = \| T f \|_2^2 ~,
\end{eqnarray}
where the last equality used the measure decomposition property. 
Now let 
\[
r_0 = \sup \left\{ r \ge 0: \int_{d(x,e) \ge r} |f(x)|^2 dx > 2 \delta - \delta^2 \right\}~.
\]
$r_0 \in (0,\infty)$ is well-defined since 
\[
\int_{d(x,e) \ge r} |f(x)|^2 dx = \mathbb{P} (R \ge r) \to 0 ~,
\]
 as $r \to \infty$. 
 Now, on the one hand, the definition of $r_0$ as supremum implies for $r>r_0$
 \begin{equation} \label{eqn:deliverable_1}
  \| P_{r} f \|_2^2  = 1 - \int_{d(x,e) \ge r} |f(x)|^2 dx \ge 1-2 \delta + \delta^2 = (1-\delta)^2~. 
 \end{equation}
 
 For any $r < r_0$ one has by definition 
 \[ 2 \delta - \delta^2 \le \int_{d(x,e) \ge r}|f(x)|^2 dx = P(R \ge r) \le \frac{\mathbb{E}(R^2)}{r^2} = \frac{\| T f \|_2^2}{r^2}~, \]
 where we used the Markov inequality for the second moment, as well as equation (\ref{eqn:expectation}).
This yields equivalently 
\[ r^2 \le \frac{\| Tf \|_2^2}{2 \delta - \delta^2}
\] for all $r < r_0$, and then $r \to r_0$ yields
\[
r_0 \le \frac{\| T f\|_2}{\left(2 \delta - \delta^2\right)^{1/2}}~.
\]
\end{proof}

\begin{proof}[Proof of Theorem \ref{thm:main}]
We first show the desired inequality for the special case $a=e, b=e$. Since both sides of the inequality are homogeneous of order 4 in $f$, we may assume that $\| f \|_2 = 1$. Additionally we assume that $\| T f \|_2 < \infty$ and $\| \widehat{T} f \|_2< \infty$, since the inequality is trivial otherwise. 

By Proposition \ref{prop:main} there exists $r_0>0$ satisfying 
\begin{equation} \label{eqn:r_0}
\forall r> r_0 \,:\, \| P_{r} f \|_2 \ge 1 - \delta~, \mbox{ and } r_0 \le \frac{\| T f \|_2}{\left(2 \delta - \delta^2 \right)^{1/2}}~.
\end{equation} The same proposition applied in Fourier domain yields $s_0 >0$ satisfying
\begin{equation} \label{eqn:s_0}
\forall s > s_0 \,:\, \| Q_{s} f \|_2 \ge 1 - \delta~, \mbox{ and } s_0 \le \frac{\| \widehat{T} f \|_2}{\left(2 \delta - \delta^2 \right)^{1/2}}~.
\end{equation} It follows for $r>r_0$ and $s>s_0$ that 
\[
\| f- P_{r} Q_{s} f \|_2 \le \| f - P_{r} f \|_2 + \| P_{r} f  - P_{r} Q_{s} f \|_2 \le 2 \delta
\]
or
\begin{equation} \label{eqn:PQf}
\| P_{r} Q_{s} f \|_2 \ge 1 - 2 \delta~.
\end{equation}
But there is also a fairly standard approach to obtain an upper estimate of $\| P_{r} Q_{s} f \|_2$, exploited in a similary way in \cite{MR997928}: The convolution theorem provides that $Q_s f = f \ast G_s$, where $G_s$ is the inverse Plancherel transform of $\mathbf{1}_{\widehat{B}_s(e)}$. Accordingly, 
\[
P_r Q_s f(x) = \int_G K(x,y) f(y) dy
\]
with the integral kernel given as $K(x,y) = \mathbf{1}_{B_r(e)}(x) G_s(xy^{-1})$. It follows that $P_r Q_s$ is a Hilbert-Schmidt operator, whose operator norm obeys
\[ \| P_r Q_s \|_{op}^2 \le \| P_r Q_s \|_{HS}^2 =  |B_r(e)| \|G_s \|_2^2 = |B_{r}(e)| |\widehat{B}_{s}(e)| \,\,,
\]
where $\| \cdot \|_{HS}$ denotes the Hilbert-Schmidt norm. 

In summary we obtain 
\begin{eqnarray*}
1-2 \delta & \le & \| P_{r} Q_{s} f \|_2 \le |B_{r}(e)|^{1/2} |\widehat{B}_{s}(e)|^{1/2} \\
& \le & A^{1/2} \widehat{A}^{1/2} r^{t/2} s^{t/2}~,
\end{eqnarray*}
using homogeneous polynomial growth.
Letting $r\to r_0$ and $s \to s_0$ then results in 
\[
1-2 \delta \le  A^{1/2} \widehat{A}^{1/2} r_0^{t/2} s_0^{t/2}~.
\]
Now employing the second equations of (\ref{eqn:r_0}) and (\ref{eqn:s_0}) we obtain
\[
    1 - 2 \delta \le  A^{1/2} \widehat{A}^{1/2} \frac{\| Tf \|_2^{t/2} \|\widehat{T} f \|_2^{t/2}}{(2 \delta - \delta^2)^t}
\]
which is rearranged to 
\[
\| Tf \|_2^2 \|\widehat{T} f \|_2^2 \ge \frac{(1-2\delta)^{4/t}(2 \delta - \delta^2)^4}{A^{2/t} \widehat{A}^{2/t}}~.
\]
Maximizing the lower bound over $\delta \in (0,1/2)$ gives
\begin{equation} \label{eqn:first_case}
\forall f \in L^2(G) ~:~ \left( \int_G d(x,e)^2 |f(x)|^2 dx \int_{\widehat{G}} \widehat{d}(\xi,e)^2 |\widehat{f}(\xi)|^2 d\xi \right) \ge C_G \|f\|_2^4~,
\end{equation} with $C_G$ given by (\ref{eqn:def_C_G}).

It remains to extend the estimate to the infimum over $a,b$, which is done by a standard trick, compare the introduction of \cite{MR1448337}:
Given any $a \in G, b \in \widehat{G}$, let
\[
f_{a,b}(x) = b(x) f(ax)~,
\] which is a time-frequency shift of $f$ by $(a,b) \in G \times \widehat{G}$. Then it is straightforward to verify that 
\begin{equation} \label{eqn:shifts}
|f_{a,b}(x)|= |f(ax)|~,|\widehat{f_{a,b}(\xi)}| = |\widehat{f} (b \xi)|
\end{equation}
It follows that 
\begin{eqnarray*}
\int_G d(x,a)^2 |f(x)|^2 dx \int_{\widehat{G}} d(\xi,b)^2 |\widehat{f}(\xi)|^2 d\xi 
& = & \int_G d(x,e) |f_{a,b}(x)|^2 dx \int_{\widehat{G}} d(\xi,e)^2 |\widehat{f_{a,b}}(\xi)|^2 d\xi \\
& \ge & C_G \| f_{a,b} \|_2^4 
\end{eqnarray*}
by (\ref{eqn:first_case}). But (\ref{eqn:shifts}) clearly also entails $\| f_{a,b} \|_2 = \| f  \|_2$, and the theorem is proved. 
\end{proof}

\section{Examples}

Clearly the relevance of Theorem \ref{thm:main} depends on the availability of metrics fulfilling the requirements; each such pair of metrics then gives rise to its own uncertainty inequality. Understanding how these different inequalities relate to each other is one possible direction of further research. 

We will now present two examples of groups to which our main result is applicable. The first is the well-understood Euclidean case, and it provides an occasion to gauge the precision of our approach. 

\subsection{The Euclidean case}

If $G = \mathbb{R}^n$, one naturally obtains the dual group $\widehat{G} = \mathbb{R}^n$. If one uses the Fourier transform
\[ \widehat{f}(\xi) = \int_{\mathbb{R}^n} f(x) e^{-2 \pi i \langle x, \xi \rangle} dx\, \,, \]
then the standard Lebesgue measure serves as Haar measure on both $G$ and $\widehat{G}$. We use the Euclidean metric on both groups, $d(x,y) = | x - y |$. The metric has homogeneous polynomial growth with rate $t=n$, since 
\[ |B_r(0)| = V_n r^n\,\,, \]
where $A = V_n = \frac{\pi^{n/2}}{\Gamma(\frac{n}{2}+1)}$ is the volume of the $n$-dimensional unit ball. Hence we essentially retrieve the standard Heisenberg uncertainty inequality
\[
\inf_{a,b \in \mathbb{R}^n}
\int_{\mathbb{R}^n} |x-a|^2 |f(x)|^2 dx \int_{\mathbb{R}^n}|\xi-b|^2 |\widehat{f}(\xi)|^2 d\xi \ge 
C_G \| f \|_2^4
\]
with
\[
C_G = \frac{\alpha_n}{V_n^{4/n}}~,~ \alpha_n = \max_{\delta \in (0,1/2)} (1-2 \delta)^{4/n} (2 \delta - \delta^2)^4~.
\]
There does not seem to exist a readily available closed form solution for $\alpha_n$. However, the computation of $\alpha_1$ amounts to determining the maximizer $\delta^*$ of 
\[
h(\delta) = (1-2\delta) (2 \delta - \delta^2) = 2\delta - 5 \delta^2 + \delta^3
\] on $(0,1/2)$, and then computing $\alpha_1 = h(\delta)^4$. This leads to 
\[
\delta^* = \frac{5-\sqrt{13}}{6}~,
\] and 
\[
\alpha_1 = \left( \frac{13\sqrt{13}-35}{54} \right)^4 \approx 0.0023363875
\] Using $V_1=2$, we then get
\[
C_{\mathbb{R}^1} = \frac{\alpha_1}{16}~.
\] The comparison to the optimal lower bound  
\[ C_{\mathbb{R}^1,\mathrm{opt}} = \frac{1}{16 \pi^2} ~,\] 
yields the quotient
\[
\frac{C_{\mathbb{R}^1}}{C_{\mathbb{R}^1,\mathrm{opt}}} = \alpha_1 \pi^2 \approx 0.0230592210~.
\] Hence the lower bound provided by Theorem \ref{thm:main} is indeed suboptimal, by a factor that is substantial if not astronomical.  

We next consider the asymptotic behaviour of our estimate. To begin with, the fact that 
\[ (1-2 \delta)^{4/t} \to 1\,\,, t \to \infty \]
uniformly on each interval $(0,1/2-\epsilon)$ allows to conclude
\begin{equation} \label{eqn:lim_at}
\lim_{t \to \infty} \alpha_t = \left( \frac{3}{4} \right)^4 \,\,.
\end{equation}
By \cite[Corollary 2.8]{MR1448337}, the optimal lower bound in dimension $n$ is given by $C_{\mathbb{R}^n,\mathrm{opt}} = \frac{n^2}{16 \pi^2}$. Using $V_n = \frac{\pi^{n/2}}{\Gamma(\frac{n}{2}+1)}$ and (\ref{eqn:lim_at}), the limit behaviour of the quotient is then determined as 
\begin{eqnarray*}
\lim_{n \to \infty} \frac{C_{\mathbb{R}^n}}{C_{\mathbb{R}^n,\mathrm{opt}}} & = & \lim_{n \to \infty}
\frac{\alpha_n}{V_n^{4/n}} \frac{16 \pi^2}{n^2} \\
& = & 16 \left( \frac{3}{4} \right)^4 \lim_{n \to \infty} \left( \frac{\Gamma(\frac{n}{2}+1)^{1/n}}{n^{1/2}} \right)^4 \\ & = & 16 \left( \frac{3}{4} \right)^4 \frac{1}{e^2} \\ & \approx & 0.17128371
\end{eqnarray*}
where the second-to-last equation used the Stirling approximation to the Gamma-function. Hence the lower bound provided by Theorem \ref{thm:main} asymptotically improves, and scales like the optimal lower bound. 

\subsection{Local fields}
\label{subsect:loc_field}

Let $K$ denote a totally disconnected local field, which is a locally compact, secound countable topological field that is not connected. 
We assume familiarity with the basic properties of these fields, and refer to \cite[Chapters I,II]{taible} for the necessary background information. We write $K$ additively, and denote its neutral element by $0$. 
It is well-known that there exists a \textit{valuation} on $K$, i.e. a map $|\cdot|: K \to \mathbb{R}_0^+$ with the following properties: 
\begin{itemize}
    \item $|x|=0$ if and only if $x=0$.
    \item $|xy|=|x||y|$ for all $x,y \in \mathit{K}$.
    \item $|x+y| \leq \,\max\,\{|x|,|y|\}$ for all $x,y \in \mathit{K}$ (ultrametric inequality). 
    \item For all Borel measurable subsets $A \subset K$ and all $x \in K$, $|xA| = |x| |A|$, where $xA = \{ x y: y \in A \}$, using multiplication within $K$.
\end{itemize}
We define the invariant metric via $d(x,y) = |x-y|$ on $G$. This metric induces the topology on $K$, hence it is continuous, and in fact proper. 
Since $\mathit{K}$ is totally disconnected, it turns out that the set of absolute values is of the form $\{q^k \,:\,k \in \mathbb{Z}\} \cup \{0\}$ for a suitable prime power $q > 0$. 

We normalize Haar measure on $K$ by the requirement $|B_q(0)| = \{ x \in K : |x| \le 1 \} = 1$. Given arbitrary $r>0$ and $m \in \mathbb{Z}$ fulfilling $q^m < r \le q^{m+1}$ that 
\[ |B_r(0)|  = |B_{q^{m+1}}(0)| = q^m < r \,\] which establishes polynomial growth for $d$, with constants $A = t = 1$.

Turning to the dual group, $K$ is canonically identified with $\widehat{K}$ by the following procedure: Pick any nontrivial character $\chi \in \widehat{K}$ with kernel equal to $B_{1}(0)$. Then, given any $x \in \mathbb{K}$, let $\chi_x : y \mapsto \chi(x y)$, where $xy$ denotes the product of the field elements in $K$. It turns out that the map $x \mapsto \chi_x$ is a topological isomorphism $K \to \widehat{K}$ \cite[1.8]{taible}. Furthermore, the Plancherel formula $\| f \|_2 = \|\widehat{f}\|_2$ holds for all $f \in L^1(K) \cap L^2(K)$ \cite[II.2]{taible}, hence the Plancherel measure coincides with Haar measure, and we may regard the Plancherel transform as a unitary map $L^2(K) \to L^2(K)$. We therefore endow the Fourier transform side with the same invariant metric, and obtain the following uncertainty inequality: 
\begin{equation} \label{eqn:HI_lf}
\inf_{a,b \in K}
\int_{K} |x-a|^2 |f(x)|^2 dx \int_{K} |\xi-b|^2 |\widehat{f}(\xi)|^2 d\xi \ge 
\alpha_1 \| f \|_2^4\,\,.
\end{equation}

Note that $K$ actually violates the qualitative uncertainty principle introduced in \cite{MR1353372}: There exist nonzero functions $f$ such that both $f$ and $\widehat{f}$ have compact supports. 

In the case where $K$ is the field of $p$-adic numbers, there already exists an uncertainty relation in the literature, namely
\begin{equation} \label{eqn:HI_padic}
\int_{K} |x-a_0|^2 |f(x)|^2 dx \int_{K} |\xi-b_0|^2 |\widehat{f}(\xi)|^2 d\xi \ge 
\frac{\| f \|_2^4}{16 \pi^2}~,
\end{equation}
see \cite[Theorem 5.2]{MR2126447}. Here $a_0, b_0 \in K$ are computed as \textit{centers of gravity} in $G$ \cite[Definition 4.1]{MR2126447}, i.e. as weighted averages in $K$, with weights $|f(x)|^2$ and $|\widehat{f}(\xi)|^2$ respectively. This approach is analogous to the way the translations minimizing the left-hand side of the uncertainty inequality are determined in the real case (compare Remark \ref{rem:minimizer}), but it relies on a suitable identification of $K$ with a subset of the positive reals. Despite the terminology there is no indication in the cited source that the center of gravity indeed minimizes the concentration integral, and the left-hand side of (\ref{eqn:HI_lf}) is by definition less than or equal to the left-hand side of (\ref{eqn:HI_padic}). Note further that the scope of (\ref{eqn:HI_lf}) is strictly larger than that of its counterpart, since \cite[Theorem 5.2]{MR2126447} imposes additional constraints on the function $f$ and on the local field. 

\begin{remark} \label{ex:non_fg}
    A local field $K$ is an example of a nondiscrete lca group possessing an open compact subgroup, for example $S = B_q(0)$. The quotient metric on $K/S$ induced by the valuation is polynomially bounded, even though the discrete group $K/S$ is not finitely generated.
\end{remark}

\begin{remark} \label{ex:counter_lf}
    The function $f = \mathbf{1}_{B_1(0)}$ provides a simple example where the minimizer of 
    \[
    a \mapsto \int_{K} |x-a|^2 |f(x)|^2 dx 
    \] is not unique. 
  Indeed, the fact that the set $B_1(0)$ is invariant under translations by $x \in B_1(0)$ (it is an additive subgroup) entails 
    \[
    \int_{K} |x-a|^2 |f(x)|^2 dx = \int_{K} |x-a'|^2 |f(x)|^2 dx
    \] as soon as $a-a' \in B_1(0)$. Note that $\widehat{f} = f$, hence the same nonuniqueness phenomenon holds also for $\widehat{f}$. 
\end{remark}

\begin{remark} 
A well-known approach to the proof of the Heisenberg uncertainty principle on $\mathbb{R}^1$ is based on the commutator estimate 
\begin{equation} \label{eqn:comm_est}
\| Xf \|_2 \| Y f \|_2 \ge |\langle [X,Y] f, f \rangle |\,\,,
\end{equation}
 see e.g. \cite[Lemma 2.1]{MR1448337}.
In the real case, on can take $X$ as the multiplication operator $Xf (x) = x f(x)$ and $Y$ the differentiation operator $Yf(x) = f'(x)$. Both are densely defined selfadjoint operators om $L^2(\mathbb{R})$, and their commutator turns out to be 
\[ [X,Y] = 2 \pi i {\rm id}_{L^2} \,\,.\] Plugging this into (\ref{eqn:comm_est}) directly yields (\ref{eqn:HI}).

One might hope that in some instances the multiplication operators $T, \widehat{T}$ used in our proof of Theorem \ref{thm:main} could serve in a similar way. However, the case of local fields provides an example where this approach fails dramatically: Let $g \in L^2(K)$ denote the function with Fourier transform given by 
\[ \widehat{g} = \mathbf{1}_{B_q(0)} - \mathbf{1}_{B_{1}}(0) \,\,\] which means that $\widehat{g}$ is precisely supported in the sphere 
$S_{1}(0)$. As a consequence, $\widehat{T} g = g$. 

$g$ is then supported in the ball $B_{q^2}(0) = \{ x \in K : |x| \le q \}$. The ultrametric triangle inequality then yields for any $a \in G$ satisfying $|a| \ge q^2$ that the translate $g_a$
defined by $g_a(x) = g(x-a)$ is supported in
\[ B_{q(a)} \subset S_{|a|}(0)\,\,.  \]  Hence $T g_a = |a| g_a$. On the other hand, $\widehat{g_a}$ is still supported in $S_{1}(0)$, which proves $\widehat{T} g_a = g_a$.

In particular, 
\[ [T,\widehat{T}] g_a = 0 \]
and the lower bound provided by (\ref{eqn:comm_est}) is zero, even though the uncertainty principle holds.  
\end{remark}



\bibliographystyle{abbrv}
\bibliography{sn-bibliography-2}

\begin{thebibliography}{10}

\bibitem{MR2401600}
L.~Ambrosio, N.~Gigli, and G.~Savar\'e.
\newblock {\em Gradient flows in metric spaces and in the space of probability
  measures}.
\newblock Lectures in Mathematics ETH Z\"urich. Birkh\"auser Verlag, Basel,
  second edition, 2008.

\bibitem{MR1353372}
D.~Arnal and J.~Ludwig.
\newblock Q.{U}.{P}.\ and {P}aley-{W}iener properties of unimodular, especially
  nilpotent, {L}ie groups.
\newblock {\em Proc. Amer. Math. Soc.}, 125(4):1071--1080, 1997.

\bibitem{MR2587584}
A.~Baklouti and E.~Kaniuth.
\newblock On {H}ardy's uncertainty principle for solvable locally compact
  groups.
\newblock {\em J. Fourier Anal. Appl.}, 16(1):129--147, 2010.

\bibitem{MR780328}
M.~Benedicks.
\newblock On {F}ourier transforms of functions supported on sets of finite
  {L}ebesgue measure.
\newblock {\em J. Math. Anal. Appl.}, 106(1):180--183, 1985.

\bibitem{MR4744878}
A.~Chattopadhyay, D.~Giri, and R.~K. Srivastava.
\newblock Qualitative uncertainty principle on certain {L}ie groups.
\newblock {\em J. Aust. Math. Soc.}, 116(3):289--307, 2024.

\bibitem{MR3336090}
P.~Ciatti, M.~G. Cowling, and F.~Ricci.
\newblock Hardy and uncertainty inequalities on stratified {L}ie groups.
\newblock {\em Adv. Math.}, 277:365--387, 2015.

\bibitem{MR2355602}
P.~Ciatti, F.~Ricci, and M.~Sundari.
\newblock Heisenberg-{P}auli-{W}eyl uncertainty inequalities and polynomial
  volume growth.
\newblock {\em Adv. Math.}, 215(2):616--625, 2007.

\bibitem{MR2126447}
M.~Cui and Y.~Zhang.
\newblock The {H}eisenberg uncertainty relation in harmonic analysis on
  {$p$}-adic numbers field.
\newblock {\em Ann. Math. Blaise Pascal}, 12(1):181--193, 2005.

\bibitem{MR997928}
D.~L. Donoho and P.~B. Stark.
\newblock Uncertainty principles and signal recovery.
\newblock {\em SIAM J. Appl. Math.}, 49(3):906--931, 1989.

\bibitem{MR1448337}
G.~B. Folland and A.~Sitaram.
\newblock The uncertainty principle: a mathematical survey.
\newblock {\em J. Fourier Anal. Appl.}, 3(3):207--238, 1997.

\bibitem{MR1574130}
G.~H. Hardy.
\newblock A {T}heorem {C}oncerning {F}ourier {T}ransforms.
\newblock {\em J. London Math. Soc.}, 8(3):227--231, 1933.

\bibitem{MR551496}
E.~Hewitt and K.~A. Ross.
\newblock {\em Abstract harmonic analysis. {V}ol. {I}}, volume 115 of {\em
  Grundlehren der Mathematischen Wissenschaften}.
\newblock Springer-Verlag, Berlin-New York, second edition, 1979.
\newblock Structure of topological groups, integration theory, group
  representations.

\bibitem{MR2122203}
E.~Matusiak, M.~\"Ozayd\i~n, and T.~Przebinda.
\newblock The {D}onoho-{S}tark uncertainty principle for a finite abelian
  group.
\newblock {\em Acta Math. Univ. Comenian. (N.S.)}, 73(2):155--160, 2004.

\bibitem{MR1181081}
R.~Meshulam.
\newblock An uncertainty inequality for groups of order {$pq$}.
\newblock {\em European J. Combin.}, 13(5):401--407, 1992.

\bibitem{MR564688}
S.~A. Morris.
\newblock Duality and structure of locally compact abelian groups{$\ldots $}for
  the layman.
\newblock {\em Math. Chronicle}, 8:39--56, 1979.

\bibitem{MR3265736}
B.~Ricaud and B.~Torr\'esani.
\newblock A survey of uncertainty principles and some signal processing
  applications.
\newblock {\em Adv. Comput. Math.}, 40(3):629--650, 2014.

\bibitem{MR1038803}
W.~Rudin.
\newblock {\em Fourier analysis on groups}.
\newblock Wiley Classics Library. John Wiley \& Sons, Inc., New York, 1990.
\newblock Reprint of the 1962 original, A Wiley-Interscience Publication.

\bibitem{taible}
M.~H. Taibleson.
\newblock {\em Fourier analysis on local fields}, volume~15.
\newblock Princeton University Press, 2015.

\bibitem{MR4229152}
A.~Wigderson and Y.~Wigderson.
\newblock The uncertainty principle: variations on a theme.
\newblock {\em Bull. Amer. Math. Soc. (N.S.)}, 58(2):225--261, 2021.

\end{thebibliography}



\end{document}